\documentclass[12pt,a4paper,oneside]{amsart}
\usepackage[a4paper]{geometry}
\usepackage{mathptmx} 
\DeclareMathAlphabet{\mathcal}{OMS}{cmsy}{m}{n} 

\usepackage{mathrsfs,stmaryrd,mathbbol}

\usepackage{perpage}

\MakePerPage[2]{footnote} 

\newtheorem*{theorem}{Theorem}
\newtheorem*{lemma}{Lemma}
\newtheorem{corollary}{Corollary}

\def\liebrack  {\ensuremath{[\,\cdot\, , \cdot\,]}} 
\def\curlybrack{\ensuremath{\{\,\cdot\, , \cdot\,\}}}
\def\form      {\ensuremath{(\,\cdot\, , \cdot\,)}}
\def\bform     {\ensuremath{\Lparen\,\cdot\, , \cdot\,\Rparen}}

\newcommand{\set}[2]{\ensuremath{\{ #1 \>|\> #2 \}}}

\begin{document}
\title[Invariant forms on generalized Jacobson--Witt algebras]
{Nonexistence of invariant symmetric forms on generalized Jacobson--Witt algebras 
revisited}
\author{Pasha Zusmanovich}
\address{Reykjav\'ik Academy, Iceland}
\email{pasha.zusmanovich@gmail.com}
\date{last minor revision June 6, 2018}
\thanks{Comm. Algebra \textbf{39} (2011), N2, 548--554; arXiv:0902.0038}

\begin{abstract}
We provide a short alternative homological argument showing that any
invariant symmetric bilinear form on simple modular generalized Jacobson--Witt algebras
vanishes, and outline another, deformation-theoretic one, allowing to describe such
forms on simple modular Lie algebras of contact type.
\end{abstract}

\maketitle

Recall that a symmetric bilinear form $\omega: L \times L \to K$ on a Lie algebra $L$ 
over a field $K$ is called \textit{invariant} if 
$$
\omega ([z,x],y) + \omega (x,[z,y]) = 0
$$
for any $x,y,z \in L$.
The linear space of all such forms is an important invariant of a Lie algebra.
For simple finite-dimensional modular Lie algebras of Cartan type, 
the description of this invariant was announced without proof in \cite{dzhu} and 
then elaborated in \cite[\S2]{dzhu-izv}, and independently in 
\cite{farnsteiner} and \cite[\S 4.6]{fs}. All these proofs are based on more or less direct computations,
employing the graded structure of the underlying Lie algebra and the Poincar\'e--Birkhoff--Witt theorem.

Here we propose a very short alternative proof for one of the four series of Lie algebras
of Cartan type -- namely, for generalized Jacobson--Witt algebras $W_n(\overline m)$ (also called
Lie algebras of the general Cartan type in the earlier literature, including the 
foundational paper \cite{kostrikin-sh}).
The proof is based on the evaluation of the second homology of a certain Lie algebra
in two ways, one of them involves the space of invariant symmetric bilinear forms
in question, and allows to treat the finite-dimensional modular and infinite-dimensional
cases in a uniform way.

\begin{theorem}
Let $A$ be an associative commutative algebra with unit over a field $K$ of characteristic
$\ne 2,3$, and $L$ be a Lie 
subalgebra of $Der(A)$ which is simultaneously a free finite-dimensional $A$-submodule
of $Der(A)$, such that 
\begin{equation}\label{hom}
Hom_A(L,A) = A \cdot \set{d(a)}{a\in A} ,
\end{equation}
where, for each $a\in A$, the map $d(a): L \to A$ is defined by the rule $d(a)(D) = D(a)$ 
for $D\in L$. Then any invariant symmetric bilinear form on $L$ vanishes.
\end{theorem}

Here and below, $Der(A)$ denotes the Lie algebra of derivations of an algebra $A$.

\begin{proof}
Let $B$ be another associative commutative algebra with unit over $K$.
We have obvious embeddings 
$$
L \simeq L \otimes 1 \subseteq Der(A) \otimes B \subseteq Der(A\otimes B) .
$$
Define another $(A\otimes B)$-submodule $\mathscr L$ of $Der(A\otimes B)$ as 
$\mathscr L = (A\otimes B) \cdot (L \otimes 1)$.
We have $\mathscr L = (A \cdot L) \otimes B = L \otimes B$, so $\mathscr L$
forms a Lie algebra, with Lie brackets defined as
$$
[D_1 \otimes b_1, D_2 \otimes b_2] = [D_1,D_2] \otimes b_1b_2
$$
for $D_1,D_2 \in L$, $b_1,b_2 \in B$.

Obviously, $\mathscr L$ is a free $A\otimes B$-module, and 
\begin{multline*}
Hom_{A\otimes B} (\mathscr L, A\otimes B) =
Hom_{A\otimes B} (L \otimes B, A \otimes B) \simeq Hom_A(L,A) \otimes B \\ =
A \cdot \set{d(a)}{a\in A} \otimes B = 
(A\otimes B) \cdot \langle d(a\otimes b) \>|\> a\in A, b\in B \rangle ,
\end{multline*}
i.e. condition (\ref{hom}) is satisfied for the pair $(\mathscr L, A\otimes B)$ too.

Now we are in position to apply to both algebras $L$ and $\mathscr L$ 
Theorem 7.1 of \cite{skr} which tells that $H_2(L,K)$ 
(the second homology of $L$ with trivial coefficients) 
is isomorphic to the first cohomology $H^1(C_A^\bullet(L,A))$ of the 
corresponding generalized de Rham complex $C_A^\bullet(L,A)$ if $\dim_A L = 1$,
and vanishes if $\dim_A L > 1$. 
Similarly, $H_2(\mathscr L, K)$ is isomorphic to 
$H^1(C_{A\otimes B}^\bullet (\mathscr L, A\otimes B))$
if $\dim_{A\otimes B} \mathscr L = 1$
and vanishes if $\dim_{A\otimes B} \mathscr L > 1$.
Note that $\dim_{A\otimes B} \mathscr L = \dim_{A\otimes B} L\otimes B = \dim_A L$.

The complex $C_A^\bullet(L,A)$ consists of $A$-multilinear Chevalley--Eilenberg
cochains $C^\bullet(L,A)$, with differential defined as in the standard Chevalley--Eilenberg 
complex. Though this and similar complexes were extensively studied in relation 
with such structures as Lie--Rinehart algebras, Lie algebroids, etc., we failed to find 
the following simple statement explicitly in the literature.

\begin{lemma}
$H^n(C_{A\otimes B}^\bullet(L\otimes B, A\otimes B)) \simeq H^n(C_A^\bullet(L,A)) \otimes B$
for any $n\in \mathbb N$.
\end{lemma}

\begin{proof}
We have: 
\begin{multline*}
Hom_{A\otimes B} ((L\otimes B)^{\otimes n}, A\otimes B) \simeq 
Hom_{A\otimes B} (L^{\otimes n} \otimes B^{\otimes n}, A\otimes B) \\ \simeq 
Hom_A (L^{\otimes n}, A) \otimes Hom_B (B^{\otimes n}, B).
\end{multline*}

It is easy to see that $Hom_B (B^{\otimes n}, B)$ consists exactly of the linear span of maps 
of the form 
$$
b_1 \otimes \dots \otimes b_n \mapsto b_1 \dots b_n b
$$
for $b_1, \dots, b_n\in B$ and some fixed $b\in B$, and hence is isomorphic to $B$.

Passing to the skew-symmetric cochains, we have
\begin{equation}\label{isom}
C_{A\otimes B}^n (L\otimes B, A\otimes B) \simeq C_A^n (L,A) \otimes B ,
\end{equation}
all $(A\otimes B)$-multilinear cochains on the left-hand side being the linear span of maps of the form 
\begin{equation}\label{yoyo}
(D_1\otimes b_1) \wedge \dots \wedge (D_n\otimes b_n) \mapsto 
\varphi(D_1, \dots, D_n) \otimes b_1 \dots b_nb ,
\end{equation}
for certain $\varphi \in C_A^n(L,A), b\in B$, and where 
$D_1, \dots, D_n \in L, b_1, ..., b_n \in B$.

Let us see how the Chevalley--Eilenberg differential $d_{L\otimes B}$ interacts with
isomorphism (\ref{isom}).
Let $\Phi \in C_{A\otimes B}^n (L\otimes B, A\otimes B)$ be determined by the formula 
(\ref{yoyo}).
Then:
\begin{align*}
d_{L\otimes B} & \Phi (D_1 \otimes b_1, \dots, D_{n+1} \otimes b_{n+1}) 
\\ &= 
\>\>\>\>\> \sum_{i=1}^{n+1} \>\>\>\>\>\>\>\> (-1)^{i+1} (D_i \otimes b_i) 
\Phi(D_1 \otimes b_1, \dots, \widehat{D_i \otimes b_i}, \dots, D_{n+1} \otimes b_{n+1}) \\ &+
\sum_{1 \le i < j \le n+1} (-1)^{i+j}
\Phi ([D_i \otimes b_i, D_j \otimes b_j], D_1 \otimes b_1, \dots, \widehat{D_i \otimes b_i},
\dots, \widehat{D_j \otimes b_j}, \\ & \qquad\qquad\qquad\qquad\qquad\> \dots, D_{n+1} \otimes b_{n+1})
\\ &=
\>\>\>\>\> \sum_{i=1}^{n+1} \>\>\>\>\>\>\>\> (-1)^{i+1} 
D_i \varphi(D_1, \dots, \widehat{D_i}, \dots, D_{n+1}) \otimes 
b_i b_1 \dots \widehat{b_i} \dots b_{n+1} b \\ 
&+
\sum_{1 \le i < j \le n+1} (-1)^{i+j}
\varphi ([D_i,D_j], D_1, \dots, \widehat{D_i}, \dots, \widehat{D_j}, \dots, D_{n+1}) 
\\ & \qquad\qquad\qquad\qquad\quad \otimes b_i b_j b_1 \dots \widehat{b_i} \dots \widehat{b_j} \dots b_{n+1} b 
\\ & \qquad\qquad\qquad\qquad\qquad\qquad\qquad\qquad =
d_L \varphi(D_1, \dots, D_{n+1}) \otimes b_1 \dots b_{n+1} b \>,
\end{align*}
where $d_L$ is the Chevalley--Eilenberg differential in the complex $C_A^\bullet (L,A)$,
and the statement of the lemma follows.
\end{proof}

\noindent\textit{Continuation of the proof of Theorem}.
On the other hand, the equality $[aD,D] = D(a)D$ for any $D\in L, a\in A$, together
with the freeness of $L$ over $A$, implies that $[L,L] = L$, 
hence $[\mathscr L, \mathscr L] = \mathscr L$, and by \cite[Theorem 0.1]{asterisque}
\begin{equation}\label{h2}
H_2(\mathscr L, K) = H_2(L\otimes B,K) \simeq
\Big( H_2(L,K) \otimes B \Big) \oplus \Big( \mathcal B(L)\otimes HC_1(B) \Big) ,
\end{equation}
where $\mathcal B(L)$ denotes the space of symmetric coinvariants 
$$
\frac{L \vee L}{\set{[z,x] \vee y + x \vee [z,y]}{x,y,z \in L}} ,
$$
and $HC_1(B)$ denotes the first cyclic homology of $B$.
The space of invariant symmetric bilinear forms is dual to $\mathcal B(L)$,
so it is sufficient to prove the vanishing of $\mathcal B(L)$.

If $\dim_A L = 1$, then 
\begin{equation}\label{yaya}
H_2(\mathscr L, K) \simeq H^1(C_{A\otimes B}^\bullet (L\otimes B, A\otimes B)) \simeq
H^1(C_A^\bullet (L,A)) \otimes B \simeq H_2(L,K) \otimes B ,
\end{equation}
where the second isomorphism follows from the lemma.
It is obvious that this isomorphism is functorial, and from the proof of Theorem 0.1
in \cite{asterisque} it follows that the isomorphism (\ref{h2}) is functorial, too.
The comparison of (\ref{h2}) and (\ref{yaya}) entails the vanishing of 
$\mathcal B(L) \otimes HC_1(B)$.

If $\dim_A L > 1$, the whole expression at the right side of (\ref{h2}) vanishes, and, 
in particular, $\mathcal B(L) \otimes HC_1(B)$ vanishes. 

It remains to pick an algebra $B$ such that $HC_1(B)$ does not vanish: for example,
$B = K1 \oplus N$, where $N$ is an algebra with trivial multiplication of dimension $>1$.
Elementary calculation shows that $HC_1(B) \simeq N \wedge N$. Another example, relevant
to modular Lie algebras, is $B = K[x]/(x^p)$, where $p>0$ is the characteristic of the ground
field $K$. Again, either elementary calculation, or reference to 
\cite[Corollary 5.4.17]{loday} shows that $\dim HC_1(K[x]/(x^p)) = 1$.
\end{proof}

\begin{corollary}[\protect{\cite[Theorem 1]{dzhu}, \cite[Corollary in \S2]{dzhu-izv}
and \cite[Theorem 4.2]{farnsteiner}}]\label{cor-1}
Any invariant symmetric bilinear form on a finite-dimensional simple generalized
Jacobson--Witt algebra over a field of characteristic $>3$, vanishes.
\end{corollary}

\begin{proof}
It is well-known that such Lie algebras satisfy
the condition of the theorem, being freely generated by the special
derivations over a divided powers algebra (see, for example, \cite[\S 4.2]{fs}).
The condition (\ref{hom}) is verified immediately.
\end{proof}

\begin{corollary}[\protect{partially implicit in \cite[\S 1]{dzhu-inf}}]\label{cor2}
Any invariant symmetric bilinear form on an infinite-dimensional simple one-sided or
two-sided Witt algebra over a field of characteristic zero, vanishes.
\end{corollary}

Here, by a Witt algebra we mean the Lie algebra of all derivations of a certain
associative commutative polynomial-like algebra in a finite number of variables,
which is the algebra of ordinary polynomials or power series in the case of the one-sided 
Witt algebra, and the algebra of Laurent polynomials or Laurent power series
in the case of the two-sided Witt algebra. This includes, among others, the Lie algebra of smooth
vector fields on the circle which is prominent in topology, physics, and other areas.

\begin{proof}[Proof of Corollary \ref{cor2}]
Analogously, any such Lie algebra is a free module over the respective commutative
associative algebra, and the condition (\ref{hom}) is obvious.
\end{proof}

Unfortunately, the proof of the theorem and its corollaries cannot be easily extended
to other types of Lie algebras of Cartan type (where it is applicable). 
Our proof is based on two results about the second homology of some kinds of
algebras. The first one -- about current Lie algebras from \cite{asterisque} -- holds 
for arbitrary Lie algebras, but the second one -- about Lie algebras of derivations from
\cite{skr} -- holds only for particular type of algebras, what stipulates the conditions
imposed in the theorem.
It is possible that some of Skryabin's theory could be developed for analogs
of Lie algebras of other Cartan types, but such extension appears to be far from 
obvious: one could easily observe that the condition 
of freeness of $L$ as an $A$-module, which is violated for
Lie algebras of Cartan types other than Jacobson--Witt algebras, is crucial there.

However, even with these results at hand, one may try to pursue a different, more 
sophisticated homological approach. Let $L$ be a simple Lie algebra
of Cartan type embedded into the Jacobson--Witt algebra $W_n(\overline m)$. Then, for any 
associative commutative algebra with unit $B$, 
$L\otimes B$ is a subalgebra of $W_n(\overline m) \otimes B$, and we may consider the 
corresponding Hochschild--Serre spectral sequence abutting to the homology 
$H_*(W_n(\overline m) \otimes B, K)$. In particular, some quotient of the term
$$
E_{02}^1 = H_2(L\otimes B, K) \simeq 
\Big( H_2(L,K) \otimes B \Big) \oplus \Big( \mathcal B(L) \otimes HC_1(B) \Big)
$$
contributes to 
$$
H_2(W_n(\overline m) \otimes B, K) \simeq 
\Big( H_2(W_n(\overline m),K) \otimes B \Big) \oplus 
\Big( \mathcal B(W_n(\overline m)) \otimes HC_1(B) \Big)
$$
(we used the formula (\ref{h2}) here).
One may show that under some additional conditions on the embedding 
$L \hookrightarrow W_n(\overline m)$ and on $B$, the relevant differentials of low 
degree in the Hochschild-Serre spectral sequence interplay well with these tensor 
product structures, and, moreover, the relevant part of $E_{02}^1$ is preserved in 
such a way that $\mathcal B(L) \otimes HC_1(B)$ maps injectively to 
$\mathcal B(W_n(\overline m)) \otimes HC_1(B)$. But since the latter vanishes, the former 
vanishes too.

However, this approach requires quite elaborated work, which does not seem 
warranted for a modest goal pursued in this note -- to give a simple alternative
proof of an already established result.
We hope to return to it, nevertheless, in another context later.

At the end, let us outline a yet another elementary approach, peculiar to simple Lie 
algebras of contact type $K_{2n+1}(m_1, \dots, m_n)$ (this approach may work also for determining other invariants of such 
algebras, such as the second cohomology with trivial coefficients). According to 
\cite[Proposition 2 on p. 263]{kostrikin-sh}\footnote{
Strictly speaking, in \cite{kostrikin-sh} the authors treat the restricted case only and remark that
the general case is similar. The general case can be found, for example, in 
\cite[pp. 2949--2950]{kuzn}.
},
any such Lie algebra can be considered as a certain filtered deformation of the 
following semidirect sum of Lie algebras:
\begin{equation}\label{s}
\Big( H_{n-1}(m_1,\dots,m_{n-1}) \otimes O_1(m_n)\Big) \inplus W_1(m_n) ,
\end{equation}
where $H_{n-1}(m_1,\dots,m_{n-1})$ is the Lie algebra of Hamiltonian type, $O_1(m_n)$ is
the reduced polynomial algebra, and $W_1(m_n)$ acts on the tensor product in a certain complicated way.

Applying: elementary considerations about invariant symmetric bilinear forms on the
semidirect sum of Lie algebras; \cite[Theorem 4.1]{asterisque} expressing
the space of symmetric bilinear forms on the current Lie algebra $L \otimes A$
in terms of its tensor components; Corollary \ref{cor-1}; and the well-known fact that 
modular Lie algebras of Hamiltonian type have a single, up to scalar, 
invariant symmetric bilinear form (see the above-cited works), one can prove that 
the space of invariant symmetric bilinear forms on Lie algebras of kind (\ref{s}) is 
$1$-dimensional, explicitly writing down the single, up to scalar, form 
$\bform$.

Knowing the invariant symmetric bilinear forms on a Lie algebra, it is possible
to determine such forms on its deformation. Indeed, writing the deformed
bracket, as usual,
as
$$
\{x,y\} = [x,y] + \varphi_1(x,y)t + \varphi_2(x,y)t^2 + \dots 
$$
one sees that the condition of invariance of the form
$$
(x,y) = (x,y)_0 + (x,y)_1 t + (x,y)_2 t^2 + \dots
$$
with respect to multiplication $\curlybrack$ is equivalent to the series of equalities 
$$
([z,x],y)_n + (x,[z,y])_n + \sum_{\substack{i+j=n-1 \\ i,j\ge 0}} (\varphi_i(z,x),y)_j + (x,\varphi_i(z,y))_j 
= 0, \quad n=0,1,2,\dots\footnote{
Added April 27, 2017: this formula is (slightly) incorrect. The correct formula
is
$$
\sum_{\substack{i+j=n \\ i,j\ge 0}} (\varphi_i(z,x),y)_j + (x,\varphi_i(z,y))_j 
= 0, \quad n=0,1,2,\dots ,
$$
where $\varphi_0(x,y) = [x,y]$. I am grateful to Andrey Krutov for pointing this
out.
}
$$
The $0$th of these equalities says that $\form_0$ is an invariant symmetric bilinear form
on $L$ with multiplication $\liebrack$, while the subsequent ones can be interpreted
as obstructions to prolongation of $\form_0$ to $\form$. Utilizing the concrete 
structure of the deformation of (\ref{s}) presented in \cite{kostrikin-sh} and \cite{kuzn}, one can arrive,
via straightforward simple calculations, to the known result that the form $\bform$
may be prolonged to $K_{2n+1}(\overline m)$ if and only if $2n+1 \equiv -5 \pmod p$.

\medskip

Thanks are due to the anonymous referee for useful comments which improved the readability
of this note, and to Rolf Farnsteiner 
for kindly supplying a reprint of his paper \cite{farnsteiner}.


\begin{thebibliography}{KS}

\bibitem[D1]{dzhu} A.S. Dzhumadil'daev,
\emph{Central extensions and invariant forms of Cartan type Lie algebras of positive 
characteristic}, Funct. Anal. Appl. \textbf{18} (1984), 331--332.

\bibitem[D2]{dzhu-izv} \bysame, \emph{Generalized Casimir elements}, 
Math. USSR Izvestiya \textbf{27} (1986), 391--400.

\bibitem[D3]{dzhu-inf} \bysame, 
\emph{Central extensions of infinite-dimensional Lie algebras}, 
Funct. Anal. Appl. \textbf{26} (1992), 247--253.

\bibitem[F]{farnsteiner} R. Farnsteiner,
\emph{The associative forms of the graded Cartan type Lie algebras},
Trans. Amer. Math. Soc. \textbf{295} (1986), 417--427.

\bibitem[KS]{kostrikin-sh} A.I. Kostrikin and I.R. Shafarevich,
\emph{Graded Lie algebras of finite characteristic}, 
Math. USSR Izvestiya \textbf{3} (1969), 237--304; 
reprinted in: I.R. Shafarevich, \emph{Collected Mathematical Papers}, 
Springer, 1989, 443--510;
\emph{Works in 3 Volumes}, Vol. 3, Part II, Moscow, 1996, 31--121 (in Russian).

\bibitem[K]{kuzn} M.I. Kuznetsov, \emph{On Lie algebras of contact type},
Comm. Algebra \textbf{18} (1990), 2943--3013.

\bibitem[L]{loday} J.-L. Loday, \emph{Cyclic Homology}, 2nd ed., Springer, 1998.

\bibitem[S]{skr} S. Skryabin,
\emph{Degree one cohomology for the Lie algebras of derivations},
Lobachevskii J. Math. \textbf{14} (2004), 69--107.

\bibitem[SF]{fs} H. Strade and R. Farnsteiner,
\emph{Modular Lie Algebras and Their Representations}, Marcel Dekker, 1988.

\bibitem[Z]{asterisque} P. Zusmanovich, 
\emph{The second homology group of current Lie algebras},
Ast\'erisque \textbf{226} (1994), 435--452; \texttt{arXiv:0808.0217}.

\end{thebibliography}
\end{document}